\documentclass[a4paper,12pt,reqno]{amsart}
\usepackage{a4wide}
\usepackage{amsmath}
\usepackage{amssymb}
\usepackage{amsthm}
\usepackage{latexsym}
\usepackage{graphicx}
\usepackage[english]{babel}
\newtheorem{prop}[subsection]{Proposition}

\newtheorem{teor}[subsection]{Theorem}
\newtheorem{lema}[subsection]{Lemma}
\newtheorem{cor} [subsection]{Corollary}
\theoremstyle{definition}
\newtheorem{dfn} [subsection]{Definition}
\theoremstyle{remark}

\newtheorem{exm} [subsection]{Example}

\def\sdepth{\operatorname{sdepth}}
\def\qdepth{\operatorname{hdepth}}
\def\hdepth{\operatorname{hdepth}}
\def\depth{\operatorname{depth}}
\def\supp{\operatorname{supp}}

\def\PP{\operatorname{P}}
\def\Ann{\operatorname{Ann}}

\numberwithin{equation}{section}

\begin{document}

\title{Remarks on the Hilbert depth of squarefree monomial ideals}

\author[Silviu B\u al\u anescu and Mircea Cimpoea\c s]{Silviu B\u al\u anescu$^1$ and Mircea Cimpoea\c s$^2$}

% \author[Mircea Cimpoea\c s]{Mircea Cimpoea\c s$^1$}  % and Christian Krattenthaler
  % {Silviu B\u al\u anescu$^1$ and Mircea Cimpoea\c s$^2$}
	% and Christian Krattenthaler$^3$
\date{}

 \keywords{Stanley depth, Hilbert depth, Depth, Simplicial complex, Squarefree monomial ideal.}

 \subjclass[2020]{05A18, 06A07, 13C15, 13P10, 13F20}

 \footnotetext[1]{ \emph{Silviu B\u al\u anescu}, National University of Science and Technology Politehnica Bucharest, Faculty of
 Applied Sciences, %Department of Mathematical Methods and Models, 
 Bucharest, 060042, E-mail: silviu.balanescu@stud.fsa.upb.ro}
\footnotetext[2]{ \emph{Mircea Cimpoea\c s}, National University of Science and Technology University Politehnica Bucharest, Faculty of
Applied Sciences, %Department of Mathematical Methods and Models, 
Bucharest, 060042, Romania and Simion Stoilow Institute of Mathematics, Research unit 5, P.O.Box 1-764,
Bucharest 014700, Romania, E-mail: mircea.cimpoeas@upb.ro, mircea.cimpoeas@imar.ro}
% \footnotetext[3]{\emph{Christian Krattenthaler}, Fakult\"at f\"ur Mathematik,
% Universit\"at Wien Oskar-Morgenstern-Platz 1 A-1090 Vienna, Austria, E-mail: Christian.Krattenthaler@univie.ac.at}

\begin{abstract}
Let $K$ be a infinite field, $S=K[x_1,\ldots,x_n]$ and $0\subset I\subsetneq J\subset S$ two monomial ideals. 
In \cite{lucrare2} we proved a new formula for the Hilbert depth of $J/I$. In this paper, we illustrate how one can
use the Stanley-Reisner correspondence between (relative) simplicial complexes and (quotients of) squarefree monomial ideals,
in order to reobtain some basic properties of the Hilbert depth. More precisely, we show that 
$\depth(J/I)\leq \qdepth(J/I)\leq \dim(J/I)$. % for any monomial ideals $0\subset I\subsetneq J \subset S$.
Also, we prove that $\qdepth(I)\geq \qdepth(S/I)+1$, if $S/I$ is Cohen-Macaulay.
% or $I$ is squarefree with $\qdepth(S/I)\leq 3$. Finally, we prove that if $I$ is squarefree and $\qdepth(S/I)=4$ then $\qdepth(I)\geq 4$.
%CK: attempts to do what? This is incomplete.
%\textbf{Keywords:} Stanley depth, Hilbert depth, depth, (relative) simplicial complex,  (squarefree) monomial ideal, Cohen-Macaulay ring, Kruskal-Katona theorem.
%\textbf{MSC2020:} 05A18, 06A07, 13C15, 13P10, 13F20
\end{abstract}

\maketitle

\section{Introduction}

Let $K$ be an infinite field and $S=K[x_1,\ldots,x_n]$, the polynomial ring over $K$ in $n$ variables.
Let $0\subset I\subsetneq J\subset S$ be two monomial ideals. 
A \emph{Stanley decomposition} of $J/I$ is a decomposition of $J/I$ as a direct sum of $K$-vector spaces 
$\mathcal D: J/I = \bigoplus_{i=1}^r m_i K[Z_i]$, where $m_i\in S$ are monomials and $Z_i\subset\{x_1,\ldots,x_n\}$.
We define $\sdepth(\mathcal D)=\min_{i=1,\ldots,r} |Z_i|$ and
$$\sdepth(J/I)=\max\{\sdepth(\mathcal D):\mathcal D\text{ is 
  a Stanley decomposition of }J/I\}.$$
The number $\sdepth(J/I)$ is called the \emph{Stanley depth} of $J/I$. 
Apel \cite{apel} reformulated a conjecture first posed by Stanley in \cite{stan}, stating that
$$\sdepth(J/I)\geq \depth(J/I)$$ for any monomial ideals $0\subset I\subsetneq J\subset S$.
Duval et al. \cite{duval} disproved this conjecture for $J/I$ with $I\neq 0$. However, the problem $\sdepth(I)\geq\depth(I)$,
for all monomial ideals $I\subset S$, remains open.
 
% Herzog \cite{her} proposed the following conjecture: $$\sdepth(I)\geq \sdepth(S/I)+1,$$
% for all monomial ideals $I\subset S$. A weaker form of this conjecture is $$\sdepth(I)\geq \sdepth(S/I),$$
% for all monomial ideals $I\subset S$. 
% We also mention that Ichim et. al. \cite{ichim} proved that, using polarization, the study of Stanley depth of
% quotients of monomial ideals can be reduced to the squarefree case.

Let $0\subset I\subsetneq J\subset S$ be two squarefree monomial ideals. We consider the poset
$$P_{J/I}=\{A\subset [n]\;:\;x_A=\prod_{j\in A}x_j \in J\setminus I\} \subset 2^{[n]},$$ 
where $[n]=\{1,2,\ldots,n\}$. 

For two subsets $C\subset D\subset [n]$, we denote $[C,D]:=\{A\subset [n]\;:\;C\subset A\subset D\}$,
and we call it the interval bounded by $C$ and $D$. A partition with intervals of $P_{J/I}$ is a decomposition
$\mathcal P:\;P_{J/I}=\bigcup_{i=1}^r [C_i,D_i]$ into disjoint intervals. We let $\sdepth(\mathcal P)=\min_{i=1,\ldots,r}|D_i|$.
Herzog, Vl\u adoiu and Zheng showed in \cite{hvz} that 
$$\sdepth(J/I)=\max\{\sdepth(\mathcal P)\;:\;\mathcal P\text{ a partition with intervals of }P_{J/I}\}.$$
In particular, the Stanley depth of $J/I$ can be computed algorithmically. Rinaldo \cite{rin} implemented an algorithm in 
CoCoA \cite{cocoa} which compute $P_{J/I}$ and, therefore, the Stanley depth of $J/I$.

We consider $S=K[x_1,\ldots,x_n]$ with the standard grading. Let $M$ be a finitely generated graded $S$-module. 
We denote 
$$H_M(t)=\sum_{j\geq 0}(\dim_K M_j)t^j,$$ the Hilbert series of $M$.
Uliczka \cite{uli} introduced a new invariant associated to $M$, called Hilbert depth and denoted $\hdepth(M)$, by
$$\hdepth(M)=\max\{r\;:\;\text{There exists a f.g. graded }S\text{-module }N\text{ with }H_N(t)=H_M(t)\}.$$
He noted that, if $M=J/I$ then $\hdepth(M)\geq \sdepth(M)$. On the other hand, it is clear that $\hdepth(M)\geq \depth(M)$.
Moreover, Uliczka proved that 
\begin{equation}\label{ulita}
\hdepth(M)=\max\{r\;:\;(1-t)^rH_M(t)=\sum_{j\geq 0}a_j\text{ with }a_j\geq 0\text{ for all }j\geq 0\}.
\end{equation}
In particular, it follows that $\hdepth(M)\leq \dim(M)$. See also \cite[Theorem 1.1]{bruns}.

Now, let $0\subset I\subsetneq J\subset S$ be two squarefree monomial ideals. We denote
$$\alpha_j(J/I)=|\{A\in P_{J/I}\;:\;|A|=j\}|,\text{ for }0\leq j\leq n.$$ 
For any $q$ with $0\leq q\leq n$, we define
\begin{equation}\label{betak}
\beta_k^q(J/I)=\sum_{j=0}^k (-1)^{k-j} \binom{d-j}{k-j} \alpha_{j}(J/I).
\end{equation}
Using an inversion formula, see for instance \cite[Equation~(7) on p.~50 with $q=0$]{RiorAA}, we get
\begin{equation}\label{alfak}
\alpha_k(J/I)=\sum_{j=0}^k \binom{d-j}{k-j} \beta^d_{j}(J/I).
\end{equation}
Using the characterization of the Hilbert depth given in \eqref{ulita}, we proved in \cite[Theorem 2.4]{lucrare2} that
$$\qdepth(J/I)=\max\{q\;:\;\beta_k^q(J/I)=\sum_{j=0}^k (-1)^{k-j}\binom{q-j}{k-j}\alpha_j(J/I)\geq 0\text{ for all }0\leq k\leq q\}.$$ 
Also, in \cite{lucrare2} we noted that if $\mathcal P:\;P_{J/I}=\bigcup_{i=1}^r [C_i,D_i]$ is an interval partition with 
$\sdepth(\mathcal P)=q$ then $\beta_k^q(J/I)\geq 0$ for all $0\leq k\leq q$. Thus, in particular, we reobtain the fact that
\begin{equation}\label{qms}
\qdepth(J/I)\geq \sdepth(J/I).
\end{equation}
On the other hand, according to \cite[Lemma 2.5]{lucrare2} we have that
\begin{equation}\label{margini}
\min\{k\geq 0\;:\;\alpha_k(J/I)>0\}\leq \qdepth(J/I)\leq \max\{k\leq n\;:\;\alpha_k(J/I)>0\}.
\end{equation}

The aim of this paper is to show how, using the above combinatorial characterization of the Hilbert depth of $J/I$, we can reobtain the basic 
algebraic properties of the Hilbert depth. In order to do so, we make use of the Stanley-Reisner theory, i.e. the correspondence between (relative) 
simplicial complexes and (quotient of) squarefree monomial ideals.

%The above formalism is deeply connected to the theory of abstract simplicial complexes. 
If $\Delta\subset 2^{[n]}$ is a simplicial complex of dimension $d-1$ and $I=I_{\Delta}$ is the Stanley-Reisner ideal
associated to $\Delta$, we note in \eqref{alfad} that 
$\alpha_k(S/I)=f_{k-1}\text{ for all }0\leq k\leq d-1,$ where
$f=(f_{-1},f_0,\ldots,f_{d-1})$ is the $f$-vector of $\Delta$. Also, if $h=(h_0,h_1,\ldots,h_d)$ is the 
$h$-vector of $\Delta$, then we note in \eqref{betad} that
$\beta_k^d(S/I)=h_k\text{ for all }0\leq k\leq d$.

More generally, if $\Psi=(\Delta,\Gamma)$ is a relative simplicial complex of dimension $d-1$, 
where $\Gamma\subset\Delta\subset 2^{[n]}$ are simplicial complexes, $I=I_{\Delta}$ and $J=I_{\Gamma}$, 
then $\alpha_k(J/I)=f_{k-1}(\Psi)$ for all $0\leq k\leq d-1$; see \eqref{alfadd}. Also $\beta_k^d(S/I)=h_k(\Psi)$ for all
$0\leq k\leq d$; see \eqref{betadd}. For further details on simplicial complexes and their connection with 
commutative algebra we refer the reader to \cite{bh}. Also, we recommend \cite{adi} and \cite{stanley} for an introduction
in the theory of relative simplicial complexes.

In Theorem \ref{t11}, we prove that, for any proper squarefree monomial ideal $I\subset S$, we have
$$\depth(S/I)\leq \qdepth(S/I) \leq \dim(S/I) \leq n-1.$$
Using Theorem \ref{t11} and the fact that the $h$-vector of a Cohen-Macaulay simplicial complex $\Delta$ 
satisfies certain numerical conditions, in Theorem \ref{t1} we prove that if $I\subset S$ is 
a squarefree monomial ideal such that $S/I$ is Cohen-Macaulay, then
$$\qdepth(S/I)=\dim(S/I)=\depth(S/I)\text{ and }\qdepth(I)\geq \qdepth(S/I)+1.$$
In particular, if $I$ is a complete intersection ideal, minimally generated
by $m$ monomials, then $$\qdepth(S/I)=\depth(S/I)=\dim(S/I)=\depth(S/I)=n-m.$$
We note that, in one of the counterexample to the Stanley conjecture, Duval et. al. \cite{duval} constructed a squarefree
monomial ideal $I\subset K[x_1,\ldots,x_{16}]$ with $$3=\sdepth(S/I)<\depth(S/I)=\dim(S/I).$$
From Theorem \ref{t1}, it follows that $\qdepth(S/I)=4$ and $\qdepth(I)\geq \depth(I)=5$, see Example \ref{anti}.

In Section $3$ we extend several results from Section $2$ to quotient of squarefree monomial ideals. In Theorem \ref{t12}, we prove that
if $0\subset I\subsetneq J\subset S$ are two squarefree monomial ideals, then $\qdepth(J/I)\leq \dim(J/I)$ and that we have equality if
$J/I$ is Cohen-Macaulay. Also, in Theorem \ref{teo22} we show that $\depth(J/I)\leq \qdepth(J/I)$.

\section{Simplicial complexes and Hilbert depth}

We recall some basic facts about simplicial complexes.

A nonempty subposet $\Delta$ of $2^{[n]}$ is called a \emph{simplicial complex} if 
for any $F\in\Delta$ and $G\subset F$ then $G\in\Delta$. 

A subset $F\in\Delta$ is called face. The dimension of $F$ is $\dim(F)=|F|-1$.

The dimension of $\Delta$, denoted $\dim(\Delta)$, is the maximal dimension of a face of $F$.

 Assume that $\Delta$ has dimension $d-1$, for some integer $1\leq d\leq n$.
 The $f$-vector of $\Delta$ is $f=(f_{-1},f_0,f_1,\ldots,f_{d-1})$, where $f_i=f_i(\Delta)=$ the 
      number of faces of dimension $i$.
% \item A facet is a face which is maximal with respect to inclusion.

 We denote $\Delta^c=2^{[n]}\setminus \Delta$.

The squarefree monomial ideal
$$I:=I_{\Delta}=(x_F=\prod_{j\in F}x_j\;:\;F\in \Delta^c)\subset S,$$
is called the \emph{Stanley-Reisner} ideal of $\Delta$.

Note that $\PP_{I}=\Delta^c$ and $\PP_{S/I}=\Delta$. In particular, we have
\begin{equation}\label{alfad}
 \alpha_k(S/I)=\begin{cases} f_{k-1},& 0\leq k\leq d \\ 0,&  d+1\leq k\leq n \end{cases}
 \text{ and }\alpha_k(I)=\begin{cases} \binom{n}{k}-f_{k-1},& 0\leq k\leq d \\ \binom{n}{k},& d+1\leq k\leq n\end{cases}.
\end{equation}

 The $h$-vector of $\Delta$ is $h=(h_0,h_1,\ldots,h_{d})$, where
      $$h_k=\sum_{j=0}^k (-1)^{k-j} \binom{d-j}{k-j} f_{j-1},\text{ for all }0\leq k\leq d+1.$$
			The $h$-vector gives the coefficients of the denominator of the Hilbert-Poincare$'$ series of 
			the Stanley-Reisner ring $R:=K[\Delta]=S/I$,
			that is $$P_R(t)=\sum_{i\geq 0}\dim_K(R_i) t^i = \frac{h_0+h_1t+\cdots+h_dt^d}{(1-t)^d}.$$

From \eqref{betak} and \eqref{alfad}, we note that 
      \begin{equation}\label{betad} 
      \beta_k^d(S/I)=h_k(\Delta)\text{ for all }0\leq k\leq d.
			\end{equation}
      Moreover, for any integer $d'$ with $0\leq d'\leq d$, we have that
			\begin{equation}\label{betadp}
			\beta_k^{d'}(\Delta)=h_k(\Delta_{\leq d'-1}),\text{ where }\Delta_{\leq d'-1} = \{F\in\Delta\;:\;\dim(F)\leq d'-1\}.
			\end{equation}
			$\Delta_{\leq d'-1}$ is called the $(d'-1)$-skeleton of $\Delta$.

More concisely, we have $\beta^d(S/I)=h(\Delta)$ and $\beta^{d'}(S/I)=h(\Delta_{\leq d'-1})$.

A simplicial complex $\Delta$ is called \emph{Cohen-Macaulay} if the ring $S/I$ is Cohen-Macaulay.

We recall the following well know results:

\begin{lema}\label{dimc}(\cite[Theorem 5.1.4]{bh})
If $\Delta$ is a simplicial complex of dimension $d-1$ then 
$$\dim K[\Delta]=d.$$ 
\end{lema}

%We also recall the following result, which will be used later on.

\begin{lema}(\cite[Exercise 5.1.23]{bh})\label{minune}
With the above notations, we have that $$\depth(S/I)=\max\{d'\;:\;\Delta_{\leq d'-1}\text{ is Cohen Macaulay }\}.$$
\end{lema}

\begin{proof}
The result follows by induction on the number of faces of $\Delta$. If $\Delta$ is Cohen Macaulay, then there is nothing to prove.
Else, let $F\in \Delta$ with $\dim(F)=\dim(\Delta)$, $\Delta_1=\Delta\setminus\{F\}$ and $\Delta_2=\langle F\rangle$.
The induction step follows from the short exact sequence
$$0\to K[\Delta_1\cup \Delta_2] \to K[\Delta_1]\oplus K[\Delta_2] \to K[\Delta_1\cap \Delta_2] \to 0$$
and the Depth lemma.
\end{proof}

Given two positive integers $\ell,k$ there is a unique way
to expand $\ell$ as a sum of binomial coefficients, as follows
$$ \ell = \binom{n_k}{k}+\binom{n_{k-1}}{k-1}+\cdots+\binom{n_j}{j},\; n_k>n_{k-1}>\cdots>n_j\geq j\geq 1.$$
This expansion is constructed using the greedy algorithm, i.e. setting $n_k$ to be the maximal $n$ such that $\ell\geq\binom{n}{k}$,
replace $\ell$ with $\ell-\binom{n_k}{k}$ and $k$ with $k-1$ and repeat until the difference becomes zero.
We define $$\ell^{(k)}=\binom{n_k}{k+1}+\binom{n_{k-1}}{k}+\cdots+\binom{n_j}{j+1}.$$
We recall the following well known results, see for instance \cite[Theorem 5.1.10]{bh} and \cite[Theorem 5.1.15]{bh}.

\begin{teor}\label{cm}
Let $\Delta$ be a Cohen-Macaulay simplicial complex of dimension $d-1$,
with the $h$-vector, $h=(h_0,h_1,\ldots,h_d)$. Then
\begin{enumerate}
\item[(1)] $0\leq h_k \leq \binom{n-d+k-1}{k}\text{ for all }0\leq k\leq d$.
\item[(2)] $0\leq h_{k+1} \leq h_{k}^{(k)}$, for all $1\leq k\leq d-1$.
\end{enumerate}
\end{teor}

Without further ado, we prove our first main result:

\begin{teor}\label{t11}
Let $I\subset S$ be a proper squarefree monomial ideal. Then:
$$\depth(S/I)\leq \qdepth(S/I) \leq \dim(S/I) \leq n-1.$$
\end{teor}

\begin{proof}
%Using the properties of polarization, we can reduce to the case when $I$ is squarefree.
Let $\Delta$ be the simplicial complex with $I=I_{\Delta}$. 
If $\dim(\Delta)=d-1$ then, according to Lemma \ref{dimc}, we have that $\dim(S/I)=d \leq n-1$.
On the other hand, from \eqref{alfad} it follows that
$\max\{k\;:\;\alpha_k(S/I)>0\}=d$.
Thus, from \eqref{margini} it follows that $\qdepth(S/I)\leq \dim(S/I)$.

%As in the proof of the previous theorem, we can assume that $I=I_{\Delta}$, where $\Delta$ is a simplicial complex.
Let $d'=\depth(S/I)$. According to Lema \ref{minune}, we have that $\Delta_{\leq d'-1}$ is Cohen-Macaulay.
From \eqref{betadp} and Theorem \ref{cm} it follows that 
$$\beta_k^{d'}(S/I)\geq 0\text{ for all }0\leq k\leq d'.$$
Therefore $\qdepth(S/I)\geq d'$, as required.
\end{proof}

We recall the following combinatorial formula:

\begin{lema}\label{magic}
For any integers $0\leq k\leq d$ and $n\geq 0$ we have that
$$ \sum_{j=0}^k (-1)^{k-j} \binom{d-j}{k-j}\binom{n}{j} = \binom{n-d+k-1}{k}.$$
\end{lema}

\begin{proof}
It is a direct application of the Chu-Vandermonde summation.
\end{proof}

\begin{lema}\label{lem1}
Let $I\subset S$ be a proper squarefree monomial ideal. 
For any $0\leq k\leq d\leq n$ we have that 
$$\beta_k^d(I) = \binom{n-d+k-1}{k} - \beta_k^d(S/I).$$
\end{lema}

\begin{proof}
It follows from \eqref{betak}, Lemma \ref{magic} and the fact that $\alpha_j(I)=\binom{n}{j}-\alpha_j(S/I)$ for all $0\leq j\leq n$.
\end{proof}

\begin{lema}\label{lem11}
For any $1\leq k\leq d\leq n$ we have that 
$$\beta_{k}^{d+1}(S/I)=\beta_k^d(S/I)-\beta_{k-1}^{d}(S/I).$$
\end{lema}

\begin{proof}
It follows from \eqref{betak} and the identity 
$$\binom{d+1-j}{k-j}=\binom{d-j}{k-1-j}+\binom{d-j}{k-j}.$$
We leave the details to the reader.
\end{proof}

In the following Theorem, we tackle the Cohen-Macaulay case: 

\begin{teor}\label{t1}
Let $I\subset S$ be a proper squarefree monomial ideal such that $S/I$ is Cohen-Macaulay. Then:
\begin{enumerate} 
\item[(1)] $\sdepth(S/I)\leq \qdepth(S/I)=\dim(S/I)=\depth(S/I)$.
\item[(2)] $\qdepth(I)\geq\qdepth(S/I)+1$.
\end{enumerate}
\end{teor}

\begin{proof}
(1) The inequality $\sdepth(S/I)\leq \qdepth(S/I)$ is a particular case of \eqref{qms}.
The equality $\dim(S/I)=\depth(S/I)$ is the definition of a Cohen-Macauly ring. Now,
the conclusion follows from Theorem \ref{t11}.

(2) As in the proof of Theorem \ref{t11}, we can assume that
$I=I_{\Delta}$ where $\Delta$ is a simplicial complex of dimension $d-1$
with $d=\dim(S/I)$. From Theorem \ref{cm}(1) and \eqref{betad} we have 
$$\beta_k^d(S/I) = h_k \geq 0\text{ for all }0\leq k\leq d.$$
On the other hand, from Lemma \ref{lem1} and Lemma \ref{lem11} it follows that
\begin{equation}\label{cuc}
 \beta_{k+1}^{d+1}(I) = % \binom{n-d+k-1}{k+1} - \beta_{k+1}^{d+1}(S/I) = 
 \binom{n-d+k-1}{k+1} -\left(\beta_{k+1}^d(S/I) - \beta_k^d(S/I)\right),
 \text{for all }0\leq k\leq d-1.
\end{equation}
On the other hand, from Theorem \ref{cm}(2) and \eqref{betad} we have that
\begin{equation}\label{cucu}
\beta_{k+1}^d(S/I)-\beta_k^d(S/I) \leq \beta_k^d(S/I)^{(k)}-\beta_k^d(S/I) \leq \beta_k^d(S/I)^{(k)},\text{for all }1\leq k\leq d-1.
\end{equation}
Since, according to Theorem \ref{cm}(1), we have $\beta_k^d(S/I)\leq \binom{n-d+k-1}{k}$ for all $0\leq k\leq d$,
it follows that 
$$\beta_k^d(S/I)^{(k)}\leq \binom{n-d+k-1}{k+1}\text{ for all }1\leq k\leq d.$$
Therefore, from \eqref{cuc} and \eqref{cucu} we get
$$\beta_{k+1}^{d+1}(I) \geq 0 \text{ for all }0\leq k\leq d-1.$$
On the other hand, $\beta_0^{d+1}(I)=\alpha_0(I)=0$ and, since $\alpha_{d+1}(S/I)=0$, we have
$$\beta_{d+1}^{d+1}(I) = \binom{n-1}{d+1} - \beta_{d+1}^{d+1}(S/I) = \binom{n-1}{d+1} + \beta_d^d(S/I) \geq 0.$$
Hence $\beta_k^{d+1}(I)\geq 0$ for all $0\leq k\leq d+1$, as required.
\end{proof}

In particular, we can reprove the following well known result:

\begin{cor}

If $I\subset S$ is a complete intersection monomial ideal, minimally generated by $m$ monomials, then 
$\qdepth(S/I)=\dim(S/I)=\sdepth(S/I)=\depth(S/I)=n-m$.
\end{cor}

\begin{proof}
The fact that $\sdepth(S/I)=n-m$ was proved by Rauf, see \cite[Theorem 1.1]{asia1}.
The other equalities follows from the previous theorem.
\end{proof}

\begin{exm}\label{anti}\rm
We consider the ideal $I=(x_{13}x_{16}
 ,x_{12}x_{16}
,x_{11}x_{16}
,x_{10}x_{16}
,x_{9}x_{16}
,x_{8}x_{16}
,x_{6}x_{16}, \linebreak
x_{3}x_{16}
,x_{1}x_{16}
,x_{13}x_{15}
,x_{12}x_{15}
,x_{11}x_{15}
,x_{10}x_{15}
,x_{9}x_{15}
,x_{8}x_{15}
,x_{3}x_{15}
,x_{13}x_{14}
,x_{12}x_{14}
,x_{11}x_{14}
,x_{10}x_{14},\linebreak
x_{9}x_{14}
,x_{8}x_{14}
,x_{10}x_{13}
,x_{9}x_{13}
,x_{8}x_{13}
,x_{6}x_{13}
,x_{3}x_{13}
,x_{1}x_{13}
,x_{10}x_{12}
,x_{9}x_{12}
,x_{8}x_{12}
,x_{3}x_{12}
,x_{10}x_{11}
,x_{9}x_{11},\linebreak
x_{8}x_{11}
,x_{6}x_{10}
,x_{3}x_{10}
,x_{1}x_{10}
,x_{3}x_{9}
,x_{5}x_{7}
,x_{3}x_{7}
,x_{2}x_{7}
,x_{1}x_{7}
,x_{5}x_{6}
,x_{2}x_{6}
,x_{1}x_{6}
,x_{4}x_{5}
,x_{3}x_{5}
,x_{1}x_{4}, \linebreak
x_{4}x_{15}x_{16}
,x_{2}x_{15}x_{16}
,x_{2}x_{4}x_{15}
,x_{6}x_{7}x_{14}
,x_{1}x_{5}x_{14}
,x_{4}x_{12}x_{13}
,x_{2}x_{12}x_{13}
,x_{2}x_{4}x_{12}
,x_{6}x_{7}x_{11}
,x_{1}x_{5}x_{11}, \linebreak
x_{4}x_{9}x_{10}
,x_{2}x_{9}x_{10}
,x_{2}x_{4}x_{9}
,x_{6}x_{7}x_{8}
,x_{1}x_{5}x_{8}) \subset S=K[x_1,\ldots,x_{16}]$.

According to \cite[Theorem 3.5]{duval}, we have that $$3=\sdepth(S/I)<\depth(S/I)=\dim(S/I)=4.$$ 
Since $S/I$ is Cohen-Macaulay, from Theorem \ref{t1}, we deduce that $\qdepth(S/I)=4$ and $\qdepth(I)\geq 5$. 
Indeed, we can compute, using CoCoA \cite{cocoa}, these invariants easily:
$$\alpha_0(S/I)=1,\;\alpha_1(S/I)=16,\;\alpha_2(S/I)=71,\;\alpha_3(S/I)=98,\;\alpha_4(S/I)=42,$$
and $\alpha_k(S/I)=0$ for $k\geq 4$. We obtain $\beta^4(S/I)=(1,12,29,0,0)$ and thus $\qdepth(S/I)=4$. 
Using the identity $\alpha_k(I)=\binom{16}{k}-\alpha_k(S/I)$, for all $0\leq k\leq 16$, we deduce that:
$$\alpha_0(I)=\alpha_1(I)=0,\;\alpha_2(I)=49,\;\alpha_3(I)=462,\;\alpha_4(I)=1778\text{ and }\alpha_k(I)=\binom{16}{k}\text{ for }k\geq 5.$$
Since $\beta_4^{10}(I)=1778-7\cdot 462+28\cdot 49 = -84<0$ it follows that $\qdepth(I)\leq 9$. By straightforward computations, we
obtain $\beta^{9}(I)=(0,0,49,119,35,693,791,1745,3003,5005)$ and thus $\qdepth(I)=9$.
\end{exm}

\section{Relative simplicial complexes and Hilbert depth}

First, we recall some basic definitions and facts regarding relative simplicial complexes; see \cite{adi} and \cite{stanley}
for further details.

A \emph{relative simplicial complex} $\Psi$ is a pair $\Psi=(\Delta,\Gamma)$, where $\Gamma\subset \Delta \subset 2^{[n]}$ are two 
simplicial complexes.

 A face $F$ of $\Psi$ is a subset $F\in \Delta\setminus \Gamma$. The dimension of $F$ is $\dim(F)=|F|-1$ and
      the dimension of $\Psi$ is the maximal dimension of a face of $\Psi$.

 If $I=I_{\Delta}$ and $J=I_{\Gamma}$ are the Stanley-Reisner ideals associated to $\Delta$, respectively $\Gamma$, then 
      $K[\Psi]:=J/I$ is called the Stanley-Reisner module associated to $\Psi$. 
			Conversely, any quotient of two squarefree monomial ideals can be regarded as the Stanley-Reisner module associated to
			a relative simplicial complex.

% \item With the above notations, the Krull dimension of the $S$-module $J/I$ is $d:=\dim(J/I)=\dim(\Psi)+1$.
The $f$-vector of $\Psi$ is $f(\Psi)=f=(f_{-1},f_0,\ldots,f_{d-1})$, where $f_i$ is the number of faces of dimension $i$ of $\Psi$.
It is clear that $f_i=f_i(\Delta)-f_i(\Gamma)$ for all $-1\leq i\leq d-1$.

The $h$-vector of $\Psi$ is $h(\Psi)=h=(h_0,h_1,\ldots,h_d)$, where $h_k=\sum_{j=0}^{k}(-1)^{k-j}\binom{d-j}{k-j}f_{j-1}$ for all $0\leq k\leq d$.

As in the case of simplicial complexes, we have that
\begin{equation}\label{alfadd}
\alpha_j(J/I)=f_{j-1}\text{ for all }0\leq j\leq d\text{ and }\alpha_j(J/I)=0\text{ for }j>d,
\end{equation}
\begin{equation}\label{betadd}
\beta_k^d(J/I)=h_k\text{ for all }0\leq k\leq d\text{ and }\beta_k^{d'}(J/I)=h_k(\Psi_{\leq d'-1})\text{ for all }0\leq k\leq d'\leq d,    
\end{equation}
where $\Psi_{\leq d'-1}=(\Delta_{\leq d'-1},\Gamma_{\leq d'-1})$ is the $(d'-1)$-skeleton of $\Psi$.

Given a squarefree monomial ideal $I$ of $S$, we denote by $\Delta(I)$, the Stanley-Reisner simplicial complex associated to $I$.

Also, given a monomial $u\in S$, its support is $\supp(u)=\{x_j\;:\;x_j\mid u\}$.

\begin{lema}\label{dimrel}
Let $0\subset I\subsetneq J\subset S$ be two squarefree monomial ideals. We consider the relative simplicial complex 
$\Psi:=(\Delta(I),\Delta(J))$. We have that:
\begin{enumerate}
\item[(1)] $\dim(\Psi)=\dim(\Delta(I:J))$.
\item[(2)] $\dim(J/I)=\dim(\Psi)+1$.
\end{enumerate}
\end{lema}

\begin{proof}
(1) It is enough to show that $\Psi$ and $\Delta(I:J)$ share the same facets. In other words, 
given a monomial $u\in S$, we have to prove that: 
\begin{enumerate}
\item[(i)] If $u\in J\setminus I$ is squarefree with $x_ju\in I$ for all
$x_j\notin\supp(u)$, then $u\notin (I:J)$.
\item[(ii)] If $u \notin (I:J)$ is squarefree with $x_ju\in (I:J)$ for all $x_j\notin\supp(u)$, 
            then $u\in J$.
\end{enumerate}
(i)  
Since $u\in J$ and $x_ju\in I$ for all $x_j\notin \supp(u)$, it follows that $(x_j\;:\;x_j\notin\supp(u))\subset (I:u)$.
If $v\in S$ is a monomial with $v\mid u$ such that $v\in (I:u)$, then $uv\in I$ and, moreover, since $I$ is squarefree, 
we have $u\in I$, a contradiction. Therefore, it follows that 
$$(I:J) \subset (I:u) = (x_j\;:\;x_j\notin \supp(u)).$$
Hence $u\notin (I:J)$, as required.

(ii) Since $u \notin (I:J)$, it follows that there exists a squarefree monomial $v\in J$ such that $uv\notin I$.
     We claim that $v\mid u$ and thus $u\in J$, as required. Indeed, if $v\nmid u$ then there exists $x_j\in\supp(v)\setminus\supp(u)$.
		 Since $x_juv\in I$ and $I$ is squarefree, it follows that $uv\in I$, a contradiction.
		
(2) From (1), it follows that
$$\dim(J/I)=\dim(S/\Ann(J/I))=\dim(S/(I:J))=\dim(\Delta(I:J))+1=\dim(\Psi)+1,$$
as required.
\end{proof}

\begin{teor}\label{t12}
Let $0\subset I\subsetneq J\subset S$ be two squarefree monomial ideal. Then:
\begin{enumerate}
\item[(1)] $\qdepth(J/I)\leq \dim(J/I)$.
\item[(2)] If $J/I$ is Cohen-Macaulay, then $\qdepth(J/I)=\dim(J/I)=\depth(J/I)$.
\end{enumerate}
\end{teor}

\begin{proof}
%Using the properties of polarization, we can reduce to the case when $I$ and $J$ are squarefree. 
We may assume that $I=I_{\Gamma}$ and $J=I_{\Delta}$ with $\Gamma \subset \Delta \subset 2^{[n]}$.
Let $\Psi:=(\Delta,\Gamma)$.

(1) According to Lemma \ref{dimrel}, we have that
$d:=\dim(J/I)=\dim(\Psi)+1.$

On the other hand, from \eqref{margini} and \eqref{alfadd} it follows that
$$\qdepth(J/I) \leq \max\{k\;:\;\alpha_k(J/I)>0\}=d.$$
Hence, we get the required inequality.

(2) Assume $J/I$ is Cohen-Macaulay of dimension $d$. According to \cite[Proposition 5.1]{stanr}, 
    we have that $h_k(\Psi)\geq 0$ for all $0\leq k\leq d$.
    Therefore, from \eqref{betadd}, it follows that $\qdepth(J/I)\geq d$. 
		The conclusion follows from (1) and the definition of a Cohen-Macaulay modules.
\end{proof}

We have the following generalization of Lemma \ref{minune}:

\begin{lema}\label{minune_mare}
Let $\Psi=(\Delta,\Gamma)$ be a relative simplicial complex, $I=I_{\Delta}$, $J=I_{\Gamma}$. Then:
$$\depth(J/I)=\max\{d'\;:\; \Psi_{\leq d'-1}\text{ is Cohen Macaulay }\}.$$
\end{lema}

\begin{proof}
According to Lemma \ref{dimrel}, we have $d=\dim(K[\Psi])=\dim(\Psi)+1$.
Note that, if $K[\Psi]=J/I$ is Cohen-Macaulay, then there is nothing to prove. 

We use induction on the number of faces of $\Psi$. 
If $|\Psi|=1$, then $\Psi=(\Delta,\Delta\setminus\{F\})$, where $F\in \Delta$ is a face of dimension $d-1$.
Assume that $F=\{1,2,\ldots,d\}$. Since $x_1x_2\ldots x_d$ is the only squarefree monomial in 
$I_{\Delta\setminus\{F})\setminus I_{\Delta}$, it is easy to see that 
$K[\Psi]\cong x_1x_2\ldots x_d K[x_1,\ldots,x_d]$ and thus $K[\Psi]$ is Cohen-Macaulay of dimension $d$.

Now, assume that $\Psi=(\Delta,\Gamma)$ and $r=\depth(K[\Psi])<d$. 
We choose $F\in \Delta\setminus \Gamma$ with $\dim(F)=\dim(\Psi)$.
We let $\Psi_1=(\Delta\setminus\{F\},\Gamma)$ and $\Psi_2=(\Delta,\Delta\setminus\{F\})$.
We have that 
\begin{equation}\label{kpsi}
K[\Psi]=J/I=I_{\Gamma}/I_{\Delta} \cong I_{\Gamma}/I_{\Delta\setminus\{F\}} \oplus
I_{\Delta\setminus\{F\}}/I_{\Delta} = K[\Psi_1]\oplus K[\Psi_2].
\end{equation}
Since $K[\Psi_2]$ is Cohen-Macaulay of dimension $d$, from \eqref{kpsi} and the induction hypothesis
it follows that
$$\depth K[\Psi]=\depth K[\Psi_1]=\max\{d'\;:\;(\Psi_1)_{\leq d'-1}\text{ is Cohen Macaulay }\}.$$
On the other hand, for $d'<d$ we have that $(\Psi_1)_{\leq d'-1}=\Psi_{\leq d'-1}$ and thus we are done.
\end{proof}

\begin{teor}\label{teo22}
For any squarefree monomial ideals $0\subset I\subsetneq J\subset S$, we have that 
$$\qdepth(J/I)\geq \depth(J/I).$$
\end{teor}

\begin{proof}
The proof is similar to the proof of Theorem \ref{t11}, using \eqref{betadd} and Lemma \ref{minune_mare}.
\end{proof}

\begin{exm}\rm
Let $S = K[x_1,\ldots , x_6]$, $I = (x_1x_4x_5, x_4x_6, x_2x_3x_6)$ and $J = (x_1x_2, x_1x_5,
x_1x_6, \linebreak x_2x_3, x_2x_4, x_4x_6)$. According to \cite[Remark 3.6]{duval}, the module $J/I$ is Cohen-Macaulay
of dimension $4$, while $\sdepth(J/I)=3$. From Theorem \ref{t12} it follows that $\qdepth(J/I)=4$, which can
be easily verified by straightforward computations.
\end{exm}

\subsection*{Acknowledgements}

%We gratefully acknowledge the use of the computer algebra system Cocoa (cf. \cite{cocoa}) for our experiments.
The second author was supported by a grant of the Ministry of Research, Innovation and Digitization, CNCS - UEFISCDI, 
project number PN-III-P1-1.1-TE-2021-1633, within PNCDI III.

% \subsection*{Data availability}

% Data sharing not applicable to this article as no data sets were generated or analyzed
% during the current study.

% \subsection*{Conflict of interest}

% The authors have no relevant financial or non-financial interests to disclose.

%\newpage

\end{document}